\documentclass{amsart}
\usepackage{amsfonts,amscd}

\usepackage[arrow,matrix,curve,cmtip,ps]{xy}
\newtheorem{claim}{}[section]
\newtheorem{theorem}[claim]{Theorem}
\newtheorem{lemma}[claim]{Lemma}
\newtheorem{proposition}[claim]{Proposition}
\newtheorem{corollary}[claim]{Corollary}

\newtheorem{definition}[claim]{Definition}

\def\proclaim #1. #2\par{\medbreak
\noindent{\bf#1.\enspace}{\sl#2}\par\medbreak} \makeatother

\DeclareMathOperator{\maxten}{\otimes_{\rm max}}
\DeclareMathOperator{\minten}{\otimes_{\rm min}}
\newcommand{\minnorm}[1]{\Vert#1\Vert_{{\rm min}}}
\newcommand{\cm}[1]{C^*_{\rm max}(#1)}
\DeclareMathOperator{\cma}{C^*_{\rm max}(A)}

\newcommand{\lra}{\longrightarrow}

\DeclareMathOperator{\Cdb}{\mathbb{C}}
\DeclareMathOperator{\Ndb}{\mathbb{N}}

\DeclareMathOperator{\Bdb}{\mathbb{B}}
\DeclareMathOperator{\Ddb}{\mathbb{D}}

\DeclareMathOperator{\Kdb}{\mathbb{K}}

\begin{document}

\title{Nuclearity-related properties for
nonselfadjoint algebras}

\author[D.\ P.\ Blecher]{David P.\ Blecher}
\address{Department of Mathematics \\ University of Houston \\
Houston, TX  77204-3008\\ USA}
\email{dblecher@math.uh.edu}

\author[B.\ L.\ Duncan]{Benton L.\ Duncan}
\address{Department of Mathematics \\ North Dakota State University \\ Fargo, ND  58105-5075}
\email{benton.duncan@ndsu.edu}

\begin{abstract}
 In analogy with the $C^*$-algebra theory, we study variants appropriate to
 nonselfadjoint algebras
 of nuclearity, the
local lifting property, exactness, and the weak expectation
property. In addition, we study the relationships between these
notions, and how they are connected with the classical $C^*$-algebra
theory through the use of $C^*$-algebras generated by the algebra.
\end{abstract}

\maketitle

\let\text=\mbox

\section{Introduction}

The concept of {\em nuclearity} is fundamental in the study of
$C^*$-algebras.  It is often defined in terms of tensor products, as
also are the slightly less well known, but also fundamental,
properties known as (Lance's) {\em weak expectation property} (WEP),
and (Kirchberg's) {\em exactness} and {\em local lifting property}
(LLP).  These are  intimately related properties which a
$C^*$-algebra may or may not possess, and the relations between them
and ensuing theory are rich and profound. The goal of the present
paper is to find appropriate generalizations of these notions to
possibly nonselfadjoint operator algebras; and to illuminate some of
the good, the bad, and the ugly that ensues. More precisely, some of
the elegant basic implications and arguments in the $C^*$-algebraic
are still valid in the nonselfadjoint case, while others only seem
to be true for very special classes of algebras. That is, for
extremely general classes of algebras some of these properties may
be a little restrictive; and in this sense the present investigation
is not as successful as some of our  previous generalizations of
$C^*$-algebraic notions to nonselfadjoint algebras. Nonetheless, the
properties we introduce are natural and do not seem to have been
considered before.
  Moreover, en route we present several new and very
basic results of independent interest. There also seems to be some
hope, as the reader will see at points in our paper, that they may
lead in the future to a new approach to Kirchberg's famous
conjectures. Another of our motivations was to find properties that
imply that $Ext$ of the $C^*$-envelope of a nonselfadjoint algebra
is a group, and this angle will be prominent in the sequel
\cite{Bseq}.  We recall from \cite{Kir1} that the LLP is more than
intimately connected with  $Ext$ of a separable $C^*$-algebra $B$
being a group (the first implies the second, and the converse is an
open question).

The new concepts introduced here are called {\em $C^*$-nuclearity},
the {\em algebra weak expectation property} (AWEP), the {\em
homomorphism local lifting property} (HLLP), {\em
$\Bdb$-nuclearity}, and {\em subexactness}. Motivated primarily by
Kirchberg's astonishing paper \cite{Kir1}, and its operator space
sequel due to Pisier, Effros, Ruan, and others, we try to build
connections between our new variants that are similar to the
classical $C^*$-algebraic theory. Throughout, we use generated
$C^*$-algebras to relate these notions to their classical
counterpart.  For instance, if $A$ is {\em Dirichlet} (that is, a
unital algebra with $A + A^*$ is norm dense in its $C^*$-envelope),
then $A$ has each of the five properties above iff the
$C^*$-envelope of $A$ has the matching $C^*$-algebra property (with
one exception in one direction: we are not sure if $A$ having HLLP
implies that the $C^*$-envelope has LLP).

It is important for us to say that of course a significant part of
modern operator space theory is devoted to linear analogues of some
of the properties mentioned above; very strikingly some of the
above-mentioned properties and their beautiful theory generalizes to
operator spaces (see e.g.\ \cite{ER,Pisbook}). However, with one
exception the operator space versions of these properties turn out
not to be appropriate for nonselfadjoint operator algebras, at least
for the approach taken here.

Turning to notation, by an operator algebra, we mean a closed, not
necessarily selfadjoint, algebra of operators on a Hilbert space. We
will sometimes silently be using very basic principles from the
theory of operator algebras, all of which are explained in
\cite{BLM}.   An operator algebra is
 {\em unital} if it has an
identity of norm $1$, and is  {\em approximately unital} if it has a
contractive two-sided approximate identity (cai).
For simplicity we will usually assume that our operator algebras
are approximately unital, but in many of the results this restriction
is not necessary, by the usual trick of considering the
unitization.  We write $A^1$ for the unitization
 of a nonunital operator algebra (see \cite[Section 2.1]{BLM}).
A {\em
unital-subalgebra}, is a subalgebra containing the identity of the
bigger algebra.  All ideals are assumed to be two-sided and closed.
Our morphisms will be linear completely contractive homomorphisms
$\theta : A \to B$ between operator algebras.   If $\theta(1) = 1$
we say that $\theta$ is {\em unital}.  As usual, {\em UCP} means
unital and completely positive.
 A {\em $C^*$-cover} of $A$ is a $C^*$-algebra
containing a copy of $A$ completely isometrically as a subalgebra,
which is generated by this copy. There are two `universal'
$C^*$-covers of $A$, a `smallest' and a `largest': the {\em
$C^*$-envelope} $C^*_{\rm e}(A)$, and the {\em maximal
$C^*$-dilation}  $\cma$. We refer the reader to \cite{BLM} for a
discussion of these two notions; they have the extremal universal
properties which the reader would expect. A new term: we will say
that $A$ is $C^*$-split if there exists a linear complete
contraction $u : C^*_{\rm e}(A) \to \cma$ extending the identity map
on $A$.  An example of this is any {\em Dirichlet} operator algebra
(this may be seen by Arveson's Proposition 1.2.8 of \cite{SOC}).  By
the well-known `rigidity' property of $C^*_{\rm e}(A)$, it follows
that $u$ is a right inverse for the epimorphism $\cma \to C^*_{\rm
e}(A)$.  We write $\Bdb$ and $\Kdb$ for respectively the bounded and
compact linear operators on $\ell^2$. We refer the reader to
\cite[Chapter 6]{BLM} or \cite{PaP} for the tensor products of
operator algebras used here. In places we also use notation and
results from the paper \cite{BR} on extensions of nonselfadjoint
algebras. By an {\em extension in the sense of} \cite{BR}, we will
mean a short  exact sequence
$$0 \lra A\overset{\alpha}\lra B\overset{\beta}\lra C \lra 0$$
of nontrivial operator algebras, with $A$ approximately unital; and
$\alpha, \beta$ are completely contractive homomorphisms, with
$\alpha$ completely isometric, and $\beta$ a complete quotient map.
Applications of our work to the theory of extensions will be
presented in a forthcoming sequel \cite{Bseq} to \cite{BR}.

\section{$C^*$ and $\Bdb$-nuclearity} \label{cnucl}

\begin{definition} An operator algebra $A$ is $C^*$-nuclear (resp.\
$\Bdb$-nuclear) if \[ A \maxten B = A \minten B \] for every
$C^*$-algebra $B$ (resp.\ for $B = \Bdb = B(\ell^2)$).
\end{definition}

It is important to note that if $A$ is not selfadjoint then allowing
$B$ in the definition of $C^*$-nuclear to be nonselfadjoint yields
a vacuous class, by  \cite[Corollary 7.1.8]{BLM}.

In 6.2.5 of \cite{BLM} it is remarked that any Dirichlet uniform
algebra, such as the disk algebra $A(\Ddb)$, is $C^*$-nuclear in our
terminology. Of course if $A$ is a $C^*$-algebra then $A$ is
$C^*$-nuclear if and only if $A$ is nuclear, and $A$ is
$\Bdb$-nuclear if and only if $A$ has the local lifting property
(LLP), see \cite{Kir1,Pisbook}.  As in the $C^*$-algebra case, we
will see that $C^*$-nuclearity implies the other new properties
mentioned in the introduction, with the possible exception of
subexactness ($C^*$-nuclearity does imply exactness).

Since we will use it many times we restate Lemma 2.8 of \cite{BR}:

\begin{lemma} \label{28} For an approximately unital operator
algebra $A$, and a $C^*$-algebra $B$, we have $A \maxten B \subset
C^*_{\rm max}(A \maxten B) = C^*_{\rm max}(A) \maxten B$ completely
isometrically.
\end{lemma}

It is not important that the Hilbert space appearing in the
definition of $\Bdb$-nuclearity be separable:

\begin{lemma} \label{cbh} An approximately unital operator
algebra $A$ is $\Bdb$-nuclear if and only if $A \maxten B(H) = A
\minten B(H)$ for every infinite dimensional Hilbert space $H$.
\end{lemma}

\begin{proof} Note that since $\Bdb$ can be embedded as a
complemented corner of $B(H)$ we see from Lemma \ref{28} and e.g.\
\cite[6.1.10]{BLM} that we have canonical complete isometries \[ A
\maxten \Bdb \subset \cma \maxten \Bdb \subset \cma \maxten B(H) =
\cm{A \maxten B(H)}.\] Thus $A \maxten \Bdb \subset A \maxten B(H)$.
Clearly $A \minten \Bdb \subset A \minten B(H)$ and thus the reverse
implication holds.

For the forward direction suppose $A$ is $\Bdb$-nuclear, and let $u
= \sum_{k=1}^n \, a_k \otimes T_k$ for $a_k \in A, T_k \in B(H)$.
Let $D$ be the $C^*$-algebra generated by $1$ and the $T_k$.  This
is separable and so there is a unital $*$-isomorphism $\pi$ from $D$
onto a $C^*$-algebra in $\Bdb$, carrying $T_k$ to $S_k$ say. The
inverse of this $*$-isomorphism extends to a UCP map $\theta : \Bdb
\to B(H)$.
  Then $$\Vert u \Vert_{A \maxten B(H)}
= \Vert \sum_{k=1}^n \, a_k \otimes \theta(S_k) \Vert_{A \maxten
B(H)} \leq \Vert \sum_{k=1}^n \, a_k \otimes S_k \Vert_{A \maxten
\Bdb}. $$ Since $A$  is $\Bdb$-nuclear, the last norm equals
$$\Vert \sum_{k=1}^n \, a_k \otimes \pi(T_k)  \Vert_{A \minten \Bdb} =
\Vert \sum_{k=1}^n \, a_k \otimes T_k \Vert_{A \minten D} = \Vert
\sum_{k=1}^n \, a_k \otimes T_k \Vert_{A \minten B(H)},$$ by
injectivity of $\minten$.  The result is now clear.
\end{proof}

One theme of our paper is that an operator space or $C^*$-algebraic
property such as  nuclearity or the LLP for a $C^*$-cover of an
algebra $A$, often says something about $C^*$-nuclearity or
$\Bdb$-nuclearity for $A$; or vice versa.  The following is a fairly
superficial result of this kind.  In the Examples section we shall
show that there are commonly met operator algebras for which $\cma$
has the LLP.

\begin{proposition}\label{LLPd} If $\cma$ is nuclear (resp.\ has the LLP)
then $A$ is $C^*$-nuclear (resp.\ is $\Bdb$-nuclear).
\end{proposition}

\begin{proof} For a $C^*$-algebra $B$ (resp.\ $B = \Bdb$)
we have canonical complete isometries \[ A \minten B \to \cma
\minten B = \cma \maxten B = \cm{A \maxten B} \] which compose to a
map whose range is in $A \maxten B$. \end{proof}

\medskip

We now look at stability of $C^*$-nuclearity and $\Bdb$-nuclearity
under the usual operator algebra constructions.

\begin{proposition}\label{maxqp} If $A$ is an approximately unital ideal in an
approximately unital operator algebra $B$, and if $D$ is any
approximately unital operator algebra, then $A \maxten D \subset B
\maxten D$ completely isometrically. \end{proposition}

\begin{proof} Take two nondegenerate commuting completely
contractive representations $\pi: A \to B(H)$ and $\theta: D \to
B(H)$.  By 2.6.13 of \cite{BLM} we can extend $\pi$ to a completely
contractive representation $\tilde{\pi}: B \to B(H)$ with \[
\theta(d) \tilde{\pi}(b) \pi(a) \zeta = \theta(d) \pi(ba) \zeta =
\pi(ba) \theta (d) \zeta = \tilde{\pi}(b) \theta(d) \pi(a) \zeta,\]
for all $a \in A, b \in B, d \in D,$ and $\zeta \in H$.  It follows
that $ \tilde{\pi}$ commutes with $\theta$ and hence the closure of
$A \otimes D $ in $B \maxten D$ will have the desired property for
$A \maxten D$ (see 6.1.1 and 6.1.11 in \cite{BLM}.)
\end{proof}

\begin{corollary}\label{ideal} If $A$ is an approximately unital ideal in a
$C^*$-nuclear (resp.\ $\Bdb$-nuclear) approximately unital
algebra $B$, then $A$ is $C^*$-nuclear (resp.\ $\Bdb$-nuclear).
\end{corollary}

{\bf Remark.}  If $A$ is a subalgebra of an operator algebra $B$,
and if $D$ is a nonselfadjoint operator algebra, then it need not be
the case that $A \maxten D \subset B \maxten D$ completely
isometrically. Note that it follows from Lemma \ref{28} that $A
\maxten D \subset \cm{A} \maxten D$ for any $C^*$-algebra $D$, but
this relation can be false if $D$ is nonselfadjoint. To see this let
$A$ be an operator algebra for which $A \minten A(\Ddb) \neq A
\maxten A(\Ddb)$ (see \cite[Corollary 7.1.8]{BLM}).  Now $A(\Ddb)$
is $C^*$-nuclear as we remarked earlier, and hence we must have
$\cma \minten A(\Ddb) = \cma \maxten A(\Ddb)$.  This contains $A
\minten A(\Ddb)$, and so it cannot contain $ A \maxten A(\Ddb)$.

\medskip

To deal with quotients we will need the following.  We do not know
if the result is true for nonselfadjoint $D$:

\begin{lemma}\label{maxq} If $D$ is a $C^*$-algebra and if $A$ is an
approximately unital ideal in an operator algebra $B$, then $(B
\maxten D) /(A \maxten D ) \cong (B/A) \maxten D$ completely
isometrically.  The same relation holds if $D$ is an approximately
unital operator algebra and $B$ is a $C^*$-algebra with ideal $A$.
\end{lemma}

\begin{proof} The assertions have similar proofs so we prove only
the first.  We know from \cite[Lemma 2.7]{BR} that $\cma$ is an
approximately unital ideal in $\cm{B}$, and that $ \cm{B}/\cma \cong
\cm{B/A}$.  Setting $C = B/A$, and using the usual $C^*$-algebra
result, we have the extension \[ 0 \lra \cma \maxten D \lra \cm{B}
\maxten D \lra \cm{C} \maxten D \lra 0. \]  We will show that \[ 0
\lra A \maxten D \lra B \maxten D \lra C \maxten D \lra 0
\] is a {\em subextension} of the first extension, in the sense of
\cite[Section 3.6]{BR}.  To see this first note that, by Lemma
\ref{28}, each term in the last sequence is a subalgebra of the
matching term in the $C^*$-algebra extension. Moreover the
intersection $(C^*_{\rm max}(A) \maxten D) \cap (B \maxten D)$ in
$C^*_{\rm max}(B)  \maxten D$ is the closure of $A \otimes D$, as
may be seen by the following argument. If $u$ is in this
intersection, and if $(e_t)$ (resp.\ $(f_s)$)
 is a cai for
 $A$ (resp.\ $D$), then $(e_t \otimes f_s) u$ is in the
 closure of $A \otimes D$.  Taking the limit, and using
 the fact that $(f_s)$
is a cai for $C^*_{\rm max}(D)$, we see that $u$ is in the
 closure of $A \otimes D$ in $C^*_{\rm max}(B)  \maxten D$.
But this closure is in the image of $A \maxten D$
since by the previous result, $A \maxten D \subset B \maxten D
\subset C^*_{\rm max}(B)  \maxten D$.  By \cite[Proposition 3.6]{BR}
we have a subextension, proving the result.
\end{proof}

\begin{lemma} \label{free} Every $C^*$-nuclear approximately unital
operator algebra is exact, hence is locally reflexive.
\end{lemma}

\begin{proof} If $D$ is a $C^*$-nuclear approximately unital
operator algebra, and if \[ 0 \lra A \lra B \lra C \lra 0 \] is a
$C^*$-extension, then by Lemma \ref{maxq} we have an exact sequence
\[ 0 \lra \cm{D} \maxten A \lra \cm{D} \maxten B \lra \cm{D} \maxten
C \lra 0 . \]  As in the proof of Lemma \ref{maxq} we have that \[ 0
\lra D \maxten A \lra D \maxten B \lra D \maxten C \lra 0 \] is a
subextension in the sense of \cite[Section 3.6]{BR}, hence it is a
`$1$-exact' sequence in the sense of \cite{ER}. Thus by
\cite[Theorem 14.4.1]{ER} we see that $D$ is exact.  \end{proof}

We are now ready for some results concerning quotients.  It is not
clear whether in the following proposition a weaker  condition on
$B/A$ than the completely contractive approximation property will
suffice.

\begin{proposition}\label{quot} Let $A$ be an approximately unital ideal in a
$C^*$-nuclear approximately unital operator algebra $B$. If $B/A$ is
separable and has the completely contractive approximation property,
then $B/A$ is $C^*$-nuclear. \end{proposition}

\begin{proof} By Proposition \ref{ideal} we know that $A$ is
$C^*$-nuclear too.  By Lemma \ref{maxq} we have $(B/A) \maxten D
\cong (B \maxten D)/(A \maxten D)$ for any $C^*$-algebra $D$. Since
$C^*$-nuclearity implies local reflexivity by Lemma \ref{free}, the
extension of $B/A$ by $A$ satisfies the condition of the lifting
theorem in \cite{ER1}.  Hence by that result there is a completely
contractive linear splitting map $B/A \to B$.  From this it follows
easily that $(B/A) \minten D \cong (B \minten D)/(A \minten D) \cong
(B/A) \maxten D$.
\end{proof}

\begin{theorem}\label{ssh} If $C$ is a $C^*$-nuclear approximately unital
operator algebra, and if
$$ 0 \lra A \lra B
\lra C \lra 0  $$ is an extension, then
for every $C^*$-algebra $D$, the associated sequence
$$0 \lra A \otimes_{\rm min} D \lra
B \otimes_{\rm min} D \lra  C \otimes_{\rm min} D$$ is an extension
(these are extension in the sense of \cite{BR}).  So $(B \otimes_{\rm min} D)/(A \otimes_{\rm min}
D) \cong (B/A) \otimes_{\rm min} D$ completely isometrically.
\end{theorem}

\begin{proof} The canonical  morphism from
$B \maxten D$ to $(B \minten D)/ (A \minten D)$ annihilates the
closure of $A \otimes B$, which by Proposition \ref{maxqp} is equal
to $A \maxten D$. Thus by Lemma \ref{maxq} we have canonical
completely contractive  morphisms
$$ C \maxten D
\cong (B \maxten D)/(A \maxten D) \to (B \minten D)/(A \minten D)
\to C \minten D $$ which compose to the identity map since $C
\minten D \cong C \maxten D$. Since the range of the first `arrow'
is dense,  we see that $C \minten D \cong (B \minten D)/(A \minten
D)$ via the canonical map. This completes the proof.
\end{proof}

{\bf Remark.}  There are several conditions equivalent to an
extension having the `tensorizing with every $C^*$-algebra' property
in the last theorem. These are studied in \cite{Bseq}.

\medskip

\begin{corollary}\label{exten} Let \[ \xymatrix{0 \ar[r] & A
\ar[r] & B \ar[r] & C \ar[r] & 0} \] be an extension of operator
algebras in the sense of {\rm \cite{BR}}. If both $A$ and $C$ are
$C^*$-nuclear (resp.\ $\Bdb$-nuclear) then $B$ is $C^*$-nuclear
(resp.\ $\Bdb$-nuclear).
\end{corollary}

\begin{proof} Let $D$ be a $C^*$-algebra.  By Lemma \ref{maxq} we have
an extension \[ \xymatrix{0 \ar[r] & A\maxten D \ar[r] & B \maxten D
\ar[r] & C \maxten D \ar[r] & 0}.
\]   By Theorem \ref{ssh} (or a variant of it in the
$\Bdb$-nuclear case) we have an extension
\[ \xymatrix{0 \ar[r] & A\minten
D \ar[r] & B \minten D \ar[r] & C \minten D   \ar[r] & 0}.
\]
Now apply the `five lemma' from \cite{BR} Lemma 3.2.  \end{proof}

\begin{corollary} \label{w16}  Let $A$ be an approximately unital
operator algebra.
Then $A$ is $C^*$-nuclear (resp.\ $\Bdb$-nuclear) if and only if
$A^1$ is $C^*$-nuclear (resp.\ $\Bdb$-nuclear) \end{corollary}

\begin{proof} We have the short exact sequence \[ 0 \lra A \lra A^1
\lra \Cdb \lra 0 \] with $A$ an approximately unital ideal in $A^1$.
The forward direction now follows from the previous result, and the
reverse follows from Corollary \ref{ideal}. \end{proof}

It is characteristic of the present paper that one gets much better
results by restricting the class of operator algebras:

\begin{proposition}\label{dirLL} If $A$ is a Dirichlet operator
algebra then $A$ is $C^*$-nuclear (resp.\ $\Bdb$-nuclear) if and
only if $C^*_{\rm e}(A)$ is nuclear (resp. has the LLP).  The {\rm
($\Leftarrow$)} implications are also true if $A$ is merely
$C^*$-split.
\end{proposition}

\begin{proof}  If $B$ is a $C^*$-algebra, then
$A \maxten B \subset C^*_{\rm max}(A) \maxten B$ by Lemma \ref{28}.
If $A$ is $C^*$-split then there is a complete contractive right
inverse to the canonical map  $C^*_{\rm max}(A) \maxten B \to
C^*_{\rm e}(A) \maxten B$.  If the latter equals $C^*_{\rm e}(A)
\maxten B$, then it follows that $A \maxten B = A \minten B$.
Conversely if the latter holds,  then $A \maxten B \subset C^*_{\rm
e}(A) \maxten B$.  If $A$ is Dirichlet, then Arveson's Proposition
1.2.8 of \cite{SOC} gives a complete isometry from the closure of $A
\minten B + A^* \minten B$ which is $C^*_{\rm e}(A) \minten B$, into
the closure of $A \maxten B + A^* \maxten B$ in $C^*_{\rm e}(A)
\maxten B$, which is $C^*_{\rm e}(A) \maxten B$.
\end{proof}

\section{The homomorphism local lifting property and $\Bdb$-nuclearity}
 \label{thehllp}

In analogy with the $C^*$-algebraic theory of the LLP, one would
expect a relationship between $\Bdb$-nuclearity and lifting
properties.  At present we only see one direction of the
relationship, which will be presented in the next theorem.

\begin{definition} An operator algebra $C$ has the homomorphism
local lifting property (HLLP) if for every operator algebra $B$, and
any approximately unital ideal $A$ in $B$, and any completely
contractive homomorphism $u: C \to B/A$ and finite dimensional
subspace $E \subset C$, there is a complete contraction from $E$ to
$B$ which is a lift of $u|_E$.
\end{definition}

Our original motivation in studying the HLLP, is that it has some
connections with the topic of when $Ext$ is a group, and this will
be presented in \cite{Bseq}. For example, we show in \cite{Bseq}
that every extension in the sense of \cite{BR} of a separable
operator algebra with the HLLP, by $\Kdb$ (or by any $C^*$-algebra
with a property described there)  is `semisplit' in the sense of
\cite{Bseq}.

\begin{theorem} \label{llp} Let $C$ be an approximately unital
operator algebra. If $C$ is $\Bdb$-nuclear then $C$ has the HLLP.
\end{theorem}

We defer the proof momentarily to prove a lemma which is of
independent interest (solving an open question about tensor products
of $M$-ideals in a special case: see the discussion before
Proposition 1.7 in \cite{AR}).

\begin{lemma} \label{toc22} Suppose that $A$ is an approximately
unital ideal in an operator algebra $B$, and that $E$ is an operator
space.  Then $A \minten E$ is a complete $M$-ideal in $B \minten E$,
and in particular is proximinal. \end{lemma}

\begin{proof}  We know that $A^{\perp \perp} = e B^{**}$ for a
central projection $e \in B^{**}$. Notice that $B \minten E$ is a
left operator $B$-module (see e.g.\ the second paragraph of the
Notes for \S 3.4 in \cite{BLM}). Thus $(B \minten E)^{**}$ is a left
dual operator $B^{**}$-module, by 3.8.9 in \cite{BLM}. Hence $e \in
B^{**}$ may be regarded as a left $M$-projection on $(B \minten
E)^{**}$, and we claim that if we do so then $(A \minten E)^{\perp
\perp} = e (B \minten E)^{**}$. It is routine to see that $(A
\minten E)^{\perp \perp} \subset e (B \minten E)^{**}$. On the other
hand, if $a_t \to e$ weak* in $B^{**}$, with $a_t \in A$, then since
$a_t (b \otimes x) \in A \otimes E \subset (A \minten E)^{\perp
\perp}$ for $b \in B$ and $x \in E$, it follows by separate weak*
continuity and density that $e (B \minten E)^{**} \subset (A \minten
E)^{\perp \perp}$.   This proves the claim, and shows that $A
\minten E$ is a complete right $M$-ideal in $B \minten E$.
Similarly, it is a complete left $M$-ideal in $B \minten E$, and so
it is a complete $M$-ideal \cite{BEZ}.\end{proof}

\begin{proof} (Of Theorem \ref{llp}) \
We will adapt a proof due to Pisier of Kirchberg's result that
$\Bdb$-nuclearity implies the LLP for $C^*$-algebras. The reader
should follow along with the proof of the (iii) implies (i) in
\cite[Theorem 16.2]{Pisbook}.  We begin with $s \in (A/I) \otimes
E^*$, as in that proof.  Suppose $s = \sum_{k=1}^n \, [a_k]  \otimes
\psi_k$, for $a_k \in A, \psi_k \in E^* \subset B(H)$.  The first
change that needs to be made is that instead of appealing to (11.1)
one uses the functoriality of the  $\maxten$ tensor product of
operator algebras. (Note: this is the only place where $u$ being a
homomorphism is used.) The appeal to Exercise 11.2 is replaced by
Lemma \ref{maxq} above. One obtains $\Vert \sum_{k=1}^n \, [a_k]
\otimes \psi_k \Vert \leq 1$ in $(A/I) \maxten B(H)$.  Then $\Vert
\sum_{k=1}^n \, a_k \otimes \psi_k \Vert \leq 1$ in $(A \maxten
B(H))/(I \maxten B(H))$, and it follows easily that $\Vert
\sum_{k=1}^n \, a_k \otimes \psi_k + (I \minten B(H)) \Vert \leq 1$
in  $(A \minten B(H))/(I \minten B(H))$. The proof of Lemma 2.4.8 in
\cite{Pisbook} may be easily adapted to our case, if one uses the
known fact (see e.g.\ \cite{DP, AR} or \cite[Proposition 6.3]{BHN})
that any ideal with cai has an approximate identity of the form
$(1-x_t)$ with $\Vert x_t \Vert \to 1$.  This implies that $\Vert
\sum_{k=1}^n \, a_k \otimes \psi_k + (I \minten E^*) \Vert \leq 1$
in $(A \minten E^*)/(I \minten E^*)$. The proof is completed as in
\cite{Pisbook} by an appeal to Lemma \ref{toc22}.
\end{proof}

 We now turn to operator algebras $A$ with the {\em unique
extension property (UEP)} considered in \cite{Dunc2}: that is, $A$
is a unital-subalgebra $A$ of a $C^*$-algebra $B$, such that for
every Hilbert space $H$ and every unital $*$-homomorphism $\pi : B
\to B(H)$, there is a unique UCP map $\Psi : B \to B(H)$ extending
$\pi_{\vert A}$. Clearly, we can replace $B(H)$ in this definition,
by `every unital $C^*$-algebra'.   It follows from the proof of
\cite[Theorem 2.7]{BL} that $B = C^*_{e}(A)$. Examples of algebras
with the UEP include logmodular and Dirichlet algebra (see p.\ {\rm
159-161} in {\rm \cite{BLM}}), and some nest algebras and crossed
products \cite{Dunc2}.  As Elias Katsoulis has pointed out, it is an
easy consequence of Choi's `multiplicative domain' trick (see e.g.\
\cite[Lemma 14.2]{Pisbook}), that any unital operator algebra
generated by unitaries will have UEP.

\begin{lemma}\label{uepten} If $A$ and $B$ are operator
algebras with the UEP, then $A \minten B$ has the UEP. \end{lemma}

\begin{proof} Since $A$ and $B$ are unital there are completely isometric
inclusions $A \subset A \minten B$ and $B \subset A \minten B$.  Now
let $\pi: C^*_{\rm e}(A \minten B) \to B(H)$ be a unital
$*$-homomorphism, and let $\sigma: C^*_{\rm e}(A \minten B) \to
B(H)$ be a unital completely positive map satisfying $ \pi|_{A
\minten B} = \tau|_{A \minten B}$. We know by \cite[Theorem
2.10]{BR} that $C^*_{\rm e}(A \minten B) = C^*_{\rm e}(A) \minten
C^*_{\rm e}(B)$. By the unique extension property  $
\sigma|_{C^*_{\rm e}(A)} = \pi|_{C^*_{\rm e}(A)}$.  A similar result
holds for $C^*_{\rm e}(B)$, and hence $C^*_{\rm e}(A)$ and $C^*_{\rm
e}(B)$ are contained in the multiplicative domain (see e.g.\
\cite{Pnbook} or \cite[Lemma 14.2]{Pisbook}) for $\sigma$. Since
$C^*_{\rm e}(A) $ and $C^*_{\rm e}(B)$ generate $C^*_{\rm e}(A
\minten B)$,we have that $C^*_{\rm e}(A \minten B)$ is contained in
the multiplicative domain for $\sigma$.  Thus $\sigma$ is a
$*$-homomorphism, and so $ \pi = \sigma$.
\end{proof}

\begin{lemma} \label{uepn} If $A$ is an operator algebra with
the UEP, and $N$ is a nonunital nuclear $C^*$-algebra, then $(A
\minten N)^1$ has the UEP. \end{lemma}

\begin{proof} Recall from \cite{BR} that $C^*_{\rm e}((N \minten A)^1)
= C^*_{\rm e}(N \minten A)^1 = (N \minten C^*_{\rm e}(A))^1$.  Let
$\theta : (N \minten C^*_{\rm e}(A))^1 \to B(H)$ be a unital
$*$-homomorphism, and let $\Phi : (N \minten C^*_{\rm e}(A))^1 \to
B(H)$ be a UCP map extending $\theta_{\vert (N \minten C^*_{\rm
e}(A))^1}$. We need to show that $\Phi = \theta$.

The restriction of $\theta$ to $N \minten C^*_{\rm e}(A)$ is of
the form $\pi \odot \rho$ for commuting $*$-representations $\pi : N
\to B(H)$ and $\rho : C^*_{\rm e}(A)  \to B(H)$.  If $(e_t)$ is an
increasing cai for $N$, then $\pi(e_t) \to q$ strongly for a
projection $q$ commuting with $\rho(1)$. Let $p = q \rho(1)$, and
identify $p B(H) p = B(K)$, where $K = p H$. Since $\pi(f) \rho(a) =
q \pi(f) \rho(a) \rho(1)$, we may replace $\pi$ and $\rho$ by $p
\pi(\cdot)$ and $p \rho(\cdot)$.
Then we
may regard $\pi$ and $\rho$ as being $B(K)$-valued with commuting
ranges, and now both of them are nondegenerate (for example
$\pi(e_t) \to p = I_K$ strongly).   Let $\pi^1 : N^1 \to B(K)$ be
the `unitization' of $\pi$, then $\pi^1 \odot \rho$ is a unital
$*$-homomorphism $N^1 \minten C^*_{\rm e}(A)  \to B(K)$. Let $\Psi$
be the restriction of $\Phi$ to $N \minten C^*_{\rm e}(A)$. Of
course $\Phi(f \otimes 1_A) = \theta(f \otimes 1_A)$ for $f \in N$,
so that it follows from  a well known principle concerning
completely positive maps, that $\Phi(fg \otimes a) = \theta(f
\otimes 1_A) \Phi(g \otimes a) = \pi(f) p \Phi(g \otimes a)$ if $f,
g \in N, a \in C^*_{\rm e}(A)$. Thus $\Psi$ may also be viewed as a
$B(K)$-valued map.

It is well known that we can extend $\Psi$ to a unital completely
positive map from the subspace $N \minten C^*_{\rm e}(A) + \Cdb
1_{N^1} \otimes 1_A$ of $N^1 \minten C^*_{\rm e}(A)$ to $B(K)$.  We
may then extend further to a unital completely positive map
$\tilde{\Psi} : N^1 \minten C^*_{\rm e}(A)  \to B(K)$. We claim that
for $f \in N^1$ we have $\tilde{\Psi}(f \otimes 1_A) = \pi^1(f)$.
Indeed for $f \in N$ we have $\tilde{\Psi}(f \otimes 1_A) = \Psi(f
\otimes 1_A) = \pi(f) = \pi^1(f)$. The claim  is also true for $f =
1_{N^1}$, and hence it is true in full generality. By the `well
known principle' used a few lines earlier, we have $\tilde{\Psi}(f
\otimes  a) = \pi^1(f) T(a)$ if $f \in N^1, a \in C^*_{\rm e}(A)$,
where $T(a) = \tilde{\Psi}(1 \otimes a)$. Note that $T$ is UCP.
Also if $a \in A$ then \[ \pi(e_t) T(a) = \tilde{\Psi}(e_t  \otimes
a) = \Phi(e_t  \otimes  a) = \pi(e_t) \rho(a),\] and taking a limit
shows that $T(a) = \rho(a)$.   Since $A$ has the UEP, $T = \rho$,
and so
\[ \Phi(f \otimes  a) = \Psi(f \otimes  a) = \tilde{\Psi}(f \otimes
a) = \pi^1(f) \rho(a) = \theta(f \otimes  a),\] for $f \in N, a \in
C^*_{\rm e}(A)$. It follows that $\Phi  = \theta$ as desired.
\end{proof}

\begin{proposition} \label{eex3} If $C$ is a separable unital
$\Bdb$-nuclear operator algebra with the UEP, then $C^*_{\rm e}(C)$
has the LLP.
\end{proposition}

\begin{proof} It is shown in \cite{Bseq}, using the HLLP
and in particular the fact mentioned above Theorem \ref{llp}, that
these hypotheses imply that $Ext_u(C^*_{\rm e}(A))$ is a group.
If  $S(C)$ for the `unitized suspension' as in \cite{Kir1}, then
$S(C)$ is a separable, unital algebra.  It is easy to see that it is
$\Bdb$-nuclear using Corollary \ref{w16}.  By the previous lemma we
also have that $S(C)$ has the UEP, so that by the above we deduce
that $Ext_u(C^*_{\rm e}(S(C)))$ is a group.  Since $C^*_{\rm
e}(S(C)) = S(C^*_{\rm e}(C))$ (see \cite[Corollary 2.11]{BR}), it
now follows from a result of Kirchberg \cite{Kir1} that $C^*_{\rm
e}(C)$ has the LLP.
\end{proof}

It follows that if $A$ is $C^*$-split, or is separable and has UEP,
then $C^*_{\rm e}(A)$ having LLP implies that $A$ has HLLP (see
Theorem \ref{llp} and Proposition \ref{dirLL}).

\section{Weak expectation} \label{awepewep}

We  turn to themes connected with the weak expectation
property.  For simplicity, unless stated otherwise we assume that
all algebras are unital, and that all subalgebras are
`unital-subalgebras'. We leave the nonunital case to the reader
using the usual unitization results (e.g.\ as in \cite[Section 2.1,
6.1.6, and 6.1.11]{BLM}, and the remark after the next definition).
This definition reduces to the WEP property of Lance if $A$ is
selfadjoint.

\begin{definition} We say that an operator algebra $A$ has the {\em
algebra weak expectation property} (AWEP) if $A \maxten D \subset B
\maxten D$ completely isometrically for every (possibly nonunital)
$C^*$-algebras $B,D$ and completely isometric embedding of $A$ as a
subalgebra of $B$. \end{definition}

\begin{proposition}  \label{cnimawep}
If an operator algebra is $C^*$-nuclear then it has the AWEP.
\end{proposition}

\begin{proof}  If $A$ is $C^*$-nuclear then for all $C^*$-algebras $D$ the
embedding $A \otimes D \to B \maxten D$ induces a tensor norm on $A
\otimes D$ which must be the maximal tensor norm.
\end{proof}

{\bf Remark.} We may assume in the last definition that $B$ is
unital, and $A$ is a unital-subalgebra of $B$.   For suppose that $A
\maxten D \subset B \maxten D$ for every unital $C^*$-algebra $B$
containing $A$ completely isometrically, and every $C^*$-algebra
$D$. If $C$ is a general $C^*$-algebra containing $A$ as a
subalgebra completely isometrically, and if $p \in C$ is the
identity for $A$, then define $B = pCp$.  The canonical projection
of $C$ onto $B$ induces canonical complete contractions
\[ A \maxten D \lra B \maxten D \lra C \maxten D \] which compose to
a complete isometry, and hence $A$ has AWEP.

\medskip

Following the $C^*$-algebraic theory (see e.g.\ \cite[Theorem
3.3]{Oz2}), we restate the
tensorial condition in terms of a `weak expectation':

\begin{proposition} \label{ozt}  If $A$ is a unital-subalgebra
of a unital $C^*$-algebra $B$, then $$ A \maxten D \subset B \maxten
D $$ completely isometrically, for every $C^*$-algebra $D$, if and
only if there is a UCP map $\varphi : B \to \cm{A}^{**}$ such that
$\varphi\vert_A$ is the `identity map' on $A$.   This is also
equivalent to: for every $C^*$-algebra $G$ containing $A$ completely
isometrically and as a unital-subalgebra, there is a UCP map
$\varphi : B \to G^{**}$ such that $\varphi\vert_A$ is the `identity
map' on $A$.
\end{proposition}

\begin{proof}
We follow the usual $C^*$-algebraic proof.  Suppose that the first
condition holds, and that $G^{**}$ is a von Neumann subalgebra $M
\subset B(H)$.  Let $D = M' \subset B(H)$.  We have $A \maxten D
\subset B \maxten D$, by the hypothesis. The canonical product map
$A \maxten D \to B(H): a \otimes d = \hat{a} d$ extends to a
completely contractive unital, and hence completely positive, map
$\Phi : B \maxten D \to B(H)$. Define $T : B \to B(H)$ by $T(x) =
\Phi(x \otimes 1)$. Since $\Phi(1 \otimes d) = d$ for all $d \in D$,
a well known lemma implies that $T(x) d = \Phi(x \otimes d) = d
T(x)$ for all $d \in D, x \in B$.  That is, $T : B \to D' = M =
G^{**}$. Finally, $T(a) =  \Phi(a \otimes 1) = \hat{a}$ for $a \in
A$.

The other direction is much easier, essentially just as in
\cite{Oz2}. Namely, we consider the canonical sequence \[ A \maxten
D \to B  \maxten D \to \cm{A}^{**}  \maxten D, \] and use the fact
that $A
\maxten D \subset \cm{A}^{**}  \maxten D$
(see  Lemma \ref{28} and \cite[Exercise 11.6]{Pisbook}).
 \end{proof}

{\bf Remark.}  In the last line of the previous result one may
replace $B$ by $B^{**}$.

\medskip
By the remark after Proposition \ref{cnimawep}, and by
Proposition \ref{ozt}, $A$ having the AWEP is equivalent to the
conditions in Proposition \ref{ozt} holding for every unital
$C^*$-algebra $B$ containing $A$ completely isometrically as a
unital-subalgebra. It is easy to see that this is equivalent to the
conditions in Proposition \ref{ozt} holding for $B = B(H)$, for all
Hilbert spaces $H$ and for all embeddings of $A$ in $B(H)$
completely isometrically as a unital-subalgebra.   One may replace
the words `for all'  in the last line with `for one fixed', to
obtain condition (iv) in the next theorem.  However this is shown
there to be equivalent to the AWEP.

\begin{theorem} \label{awep} For a unital operator algebra $A$,
consider the following conditions: \begin{enumerate}

\item[(i)] $\cma$ has the WEP.

\item[(ii)] $A$ has the AWEP.

\item[(iii)]  There exists an injective operator space $R \subset
\cma^{**}$
containing the canonical copy of $A$.
\item[(iii)']  For every $C^*$-cover $B$ of $A$, there exists an
injective operator space $R \subset B^{**}$ containing the
canonical copy of $A$.

\item[(iv)]  For every $C^*$-algebra $B$ containing $A$
completely isometrically as a unital subalgebra, there exists a
Hilbert space $H$ and a completely isometric unital homomorphism
$\pi: A \to B(H)$, and a UCP map $T: B(H) \to B^{**}$, such that $T
\circ \pi = I_A$.

\item[(v)]  $C^*_{\rm e}(A)$ has the WEP. \end{enumerate} Then ${\rm (i)}
\Rightarrow {\rm (ii)} \Leftrightarrow {\rm (iii)} \Leftrightarrow
{\rm (iii)'} \Leftrightarrow {\rm (iv)} \Rightarrow {\rm (v)}$.
\end{theorem}

\begin{proof} ${\rm (i)} \Rightarrow {\rm (iii)}$ \ If ${\rm (i)}$
holds then there is an injective operator system $R$ between $\cma$
and $\cma^{**}$.

${\rm (ii)} \Rightarrow {\rm (iv)}$ \ This is a corollary of
Proposition \ref{ozt}.

${\rm (iv)} \Rightarrow {\rm (iii)'}$ \ Suppose that  $A$ is a
unital-subalgebra of a $C^*$-algebra $B$, and that $B^{**}$ is a
unital-subalgebra of $B(K)$. By injectivity of $B(H)$ we can extend
the map $\pi$ in (iv) to a completely contractive unital, hence UCP,
map $\tilde{\pi} : B(K) \to B(H)$.  If $T$ is as in (iv), let
$\tilde{T} = T \circ \tilde{\pi} : B(K) \to B^{**}$. Let $\Theta$ be
a minimal $\hat{A}$-projection on $B(K)$.  Then $\Vert
\tilde{T}(\Theta( \cdot )) \Vert$ is a $\hat{A}$-seminorm on $B(K)$
which is dominated by $\Vert \Theta( \cdot ) \Vert$, and hence the
two must coincide. By \cite[Lemma 4.2.2]{BLM} we deduce that
$\tilde{T} \circ \Theta$ is idempotent. Since the range of
$\tilde{T} \circ \Theta$ is contained in $B^{**}$ we deduce that
there is an injective operator system $Z$ with $A \subset Z \subset
B^{**}$.

 ${\rm (iii)'} \Rightarrow
{\rm (iii)}$ \ Trivial.

${\rm (iii)} \Rightarrow {\rm (ii)}$ \ If $A$ is a unital-subalgebra
of a $C^*$-algebra $B$, we can extend the inclusion $A \subset \cma$
to a UCP map $B \to R$.  Hence $A$ satisfies the hypotheses in
Proposition \ref{ozt}, which yields ${\rm (ii)}$.

${\rm (iii)'} \Rightarrow {\rm (v)}$ \ There is an injective
 $Z$  with $A \subset Z \subset C^*_{\rm e}(A)^{**}
\subset B(H)$. Let $T : B(H) \to Z$ be a UCP idempotent.
By the rigidity property of $C^*_{\rm e}(A)$, we have that $R =
T|_{C^*_{\rm e}(A)}$ is a complete isometry onto $W = T(C^*_{\rm
e}(A))$.  Then $R^{-1}: W \to C^*_{\rm e}(A)$ extends to a complete
contraction $\mu: C^*_{\rm e}(A)^{**} \to Z \subset C^*_{\rm
e}(A)^{**}$ such that $ \mu|_{A} = I_{A}$. Then $ \mu \circ T: B(H)
\to C^*_{\rm e}(A)^{**}$ with $ \mu(T(x)) = R^{-1}(T(x)) =
R^{-1}(R(x)) = x$ for all $x \in C^*_{\rm e}(A)$.  Hence $C^*_{\rm
e}(A)$ has the WEP.
\end{proof}

{\bf Remarks.} 1) \ Variants of the above proof shows that if some
$C^*$-algebra generated by $A$ has the WEP then so does $C^*_{\rm
e}(A)$; and that $C^*_{\rm e}(A)$ has the WEP iff there exists an
injective  $R \subset C^*_{\rm e}(A)^{**}$
containing the canonical copy of $A$.

2) \  As in Proposition \ref{ozt},
one may
replace $B$ in (iv) by $C^*_{\rm max}(A)$.

\begin{proposition}\label{wenu} Let $A$ and $B$ be approximately unital
operator algebras, with $A$ $\Bdb$-nuclear and $B$ having the AWEP.
We have $ A \minten B = A \maxten B$ if either $A$ or $B$ is a
$C^*$-algebra.
\end{proposition}

\begin{proof} If $A$ is a $C^*$-algebra with the LLP and if $B$ has
the AWEP then $A \maxten B \subset A \maxten C^*_{\rm e}(B)$.  By
Theorem \ref{awep} we have $C^*_{\rm e}(B)$ has the WEP, and so by
the matching theorem of Kirchberg \cite[Proposition 1.1 (i)]{Kir1}
we have $A \minten C^*_{\rm e}(B) = A \maxten C^*_{\rm e}(B)$.  From
this the result is clear.

If $B \subset B(H)$ is a $C^*$-algebra with the WEP, and if $A$ is
$\Bdb$-nuclear, then using Lemma \ref{28} we have
\[ A \maxten
B \subset \cma \maxten B \subset \cma \maxten B(H) = \cm{A \maxten
B(H)}.\] Since $A \minten B(H) = A \maxten B(H)$ and $A \minten B
\subset A \minten B(H)$ we are done.
\end{proof}

\begin{proposition}\label{weps}  Suppose that $A$ is a
$C^*$-split unital operator algebra. Then $A$ has AWEP iff $C^*_{\rm
e}(A)$ has WEP. Also,  $A$ has AWEP if $C^*(F) \minten A = C^*(F)
\maxten A$ completely isometrically for every discrete free group
$F$.  The converse of the last assertion holds too, if $A$ is
Dirichlet.
\end{proposition}

\begin{proof}    If
$C^*_{\rm e}(A)$ has WEP then there is an injective between
$C^*_{\rm e}(A)$ and its second dual. If also $A$ is $C^*$-split
then there is an injective between $A$ and $\cma^{**}$.    The rest
of the first `iff'  follows from Theorem \ref{awep}.   The rest
follows from Kirchberg's matching result for $C^*$-algebras
\cite{Kir1} and the proof of Proposition \ref{dirLL}.
\end{proof}

{\bf Remark.}  We do not know if either direction of the last
assertion of the proposition is true for general operator algebras.
It is easy however easy to see that like the AWEP, the condition
involving $C^*(F)$ holds if $\cma$ has the WEP.  See also
Proposition \ref{var} for another result concerning this property.

\section{Some connections with exactness}

We recall that a $C^*$-algebra is  nuclear iff it is both exact and
has the WEP.   The reader familiar with Kirchberg's work on
exactness of $C^*$-algebras (see e.g.\ \cite{Kir1, Kir2}),  will
expect that we need to consider the following notion of exactness
for nonselfadjoint operator algebras. Fortunately, this coincides
with the usual operator space variant of exactness studied
 by Pisier \cite{Pis}, as we shall soon see.

\begin{definition} We say that an operator
algebra $D$ is {\em OA-exact} if for every extension  \[ 0 \lra A
\lra B \lra C \lra 0 \] in the sense of \cite{BR}, the induced
sequence
\[ 0 \lra A \minten D \lra B \minten D \lra C \minten D \lra 0 \] is
an extension.
\end{definition}

\begin{proposition} \label{next} For a subalgebra $D \subset B(H)$.
The following are equivalent:

\begin{enumerate} \item[(i)] $D$ is OA-exact.

\item[(ii)] The induced sequence \[ 0 \lra \Kdb \minten D \lra
\Bdb \minten D \lra (\Bdb/\Kdb) \minten D \lra 0 \] is an extension
in the sense of \cite{BR}.

\item[(iii)] $D$ is exact as an operator space.

\item[(iv)] If $u: D \to B(H)$ is the inclusion map, then
for every approximately unital operator algebra $A$ the map $I_A
\otimes u$ extends to a contraction from $A \minten D$ to $A \maxten
B(H)$.
\end{enumerate} \end{proposition}

\begin{proof} By \cite[Theorem 14.4.1]{Pisbook} any OA-exact
operator algebra is exact as an operator space.  Conversely, if $D$
is exact as an operator space, and if \[ 0 \lra A \lra B \lra C \lra
0
\] is an extension in the sense of \cite{BR}, then we have the
 induced sequence \[ 0 \lra \cma
\lra \cm{B} \lra \cm{C} \lra 0 \] is an extension of $C^*$-algebras
by \cite[Lemma 2.7]{BR}.
By \cite[Theorem 14.4.1(iv)]{ER} we have the extension
\[ 0 \lra \cma \minten D \lra \cm{B} \minten D \lra \cm{C} \minten D
\lra 0 \] and we can apply \cite[Proposition 3.6]{BR} and the idea
in the proof of Lemma \ref{maxq} to see that $D$ is OA-exact.

That ${\rm (ii)}$ and ${\rm (iii)}$ are equivalent follows from
e.g.\ \cite[Theorem 14.4.2]{ER}.  Further \cite[Theorem
17.1]{Pisbook} gives ${\rm (iv)} \Rightarrow {\rm (iii)}$. Finally,
assuming  (iii), recall from Lemma \ref{28} that $ A \maxten B(H)
\subset \cma \maxten B(H)$ completely isometrically.  Then applying
\cite[Theorem 17.1]{Pisbook}, the map $I_{\cma} \otimes u$ extends
to a contraction
\[ \cma \minten D \to \cma \maxten B(H). \] Restricting this map to $A
\minten D$ we obtain ${\rm (iv)}$.
\end{proof}

Any exact operator space is a subspace of an exact unital operator
algebra.  To see this let $E$ be an exact operator space and let
$\mathcal{U}(E)$ be the universal algebra for $E$ as in
2.2.10-2.2.11 of \cite{BLM}. Using the well known characterization
of exactness in terms of subspaces of $M_n$, we see that
$\mathcal{U}(E)$ is exact if $E$ is exact as an operator space by
using the following variant of \cite[Proposition 2.2.11]{BLM}. In
our case we apply the next result to both $T$ and $T^{-1}$ to see
that any finite dimensional subspace of $\mathcal{U}(E)$ can be
embedded as a subspace of $M_{2n}$.

\begin{proposition} \label{p} If $T: E \to F$ is a linear map between
operator spaces with $ \Vert T \Vert_{\rm cb} \geq 1$, then the
induced unital map $ \theta_T: \mathcal{U}(E) \to \mathcal{U}(F)$
satisfies $ \Vert \Theta_T \Vert_{\rm cb} = \Vert T \Vert_{\rm cb}$.
\end{proposition}

\begin{proof} Let $M = \Vert T \Vert_{\rm cb}$.  Then $u =
T/M$ is completely contractive, so that the map $ \theta_u$ in
\cite[Proposition 2.2.11]{BLM} is completely contractive.  But
$\theta_T = A \theta_u A^{-1}$ where $A$ is the diagonal scalar
matrix with entries $M$ and $1$.  From this it is clear that $ \Vert
\theta_T \Vert_{\rm cb} \leq M$. \end{proof}

We will say more about $\mathcal{U}(E)$ in the final subsection of
our paper.

We now consider a stronger property than exactness.  We say that an
operator algebra $A$ is {\em subexact} if it is a subalgebra of an
exact $C^*$-algebra.   We show in Section \ref{uofx} that an exact
operator algebra need not be subexact.   The following is obvious:

\begin{proposition} \label{p2}  $A$ is subexact
 if and only if $C^*_{\rm e}(A)$ is exact.
\end{proposition}

The  next two results
suggest that $C^*$-nuclearity
is not as strong a condition as might at first appear if one views
it from a `commutant lifting theorem' perspective, see
\cite[Proposition 2.5]{PaP}.

\begin{theorem} \label{ewepcnuc} A unital operator algebra $A$ is
$C^*$-nuclear iff $A$ is exact and has the AWEP.

 Also,  $A$ is
both subexact and has the AWEP, iff  both $A$ is $C^*$-nuclear and
$C^*_{\rm e}(A)$ is nuclear.
\end{theorem}

\begin{proof} We know from Lemma \ref{free} that $C^*$-nuclearity implies
exactness, and from Proposition \ref{cnimawep} we know that
$C^*$-nuclearity implies the AWEP.  Conversely, suppose $A$ is exact
and has the AWEP.  For any $C^*$-algebra $D$ we have by
\cite[(6.3)]{BLM} that $A \maxten D \subset \cma \maxten D$. Indeed
by an argument similar to that of \cite[(6.3)]{BLM},  using the
universal property of $\maxten$ and \cite[Corollary 2.5.6]{BLM}, we
have \[ A \maxten D \subset \cma \maxten D \subset \cma^{**} \maxten
D \] completely isometrically. On the other hand, the composition of
the maps in the last string agrees with the composition of the
following canonical maps (induced by the maps in Theorem \ref{awep}
(iv) with $B = \cma$):
\[ A \maxten D \to B(H) \maxten D \to \cma^{**} \maxten D. \] This
forces $A \maxten D \subset B(H) \maxten D$ completely
isometrically.  By part ${\rm (iv)}$ of Proposition \ref{next}, the
exactness of $A$ gives $A \minten D \subset B(H) \maxten D$.  It is
now easy to see that $A \maxten D = A \minten D$.  That is, $A$ is
$C^*$-nuclear.

For the second equivalence notice that if $A$ is a subalgebra of a
nuclear $C^*$-algebra $N$, then the $C^*$-algebra generated by $A$
in $N$ is exact and hence so is its quotient $C^*_{\rm e}(A)$.  If
in addition $A$ has the AWEP, then by Theorem \ref{awep} we have
that $C^*_{\rm e}(A)$ has the WEP.  Hence $C^*_{\rm e}(A)$ is
nuclear by Exercise 17.1 of \cite{Pisbook}. Since $C^*_{\rm e}(A)$
is nuclear we know that $ C^*_{\rm e}(A) \minten D = C^*_{\rm e}(A)
\maxten D$ for all $C^*$-algebras $D$.  Hence $ A \minten D \subset
C^*_{\rm e}(A) \maxten D$ completely isometrically for all
$C^*$-algebras $D$. This, by part ${\rm (iv)}$ of Proposition
\ref{next}, forces $A$ to be exact. Thus by the first chain of
equivalences we have that $A$ is $C^*$-nuclear. Finally, if
$C^*_{\rm e}(A)$ is nuclear, then $A$ is clearly subexact;
and
if $A$ is $C^*$-nuclear then $A$ has the AWEP.
\end{proof}

{\bf Remark.}  By \cite[Theorem 12.6]{Pisbook}, $A$ is $C^*$-nuclear
if and only if there is a net of finite rank contractions $v_t: A
\to M_{n_t}$ and maps $ w_t: M_{n_t} \to \cma$ with $\Vert w_t
\Vert_{{\rm dec}} \leq 1$ for all $t$, such that $w_tv_t$ converges
pointwise to the natural inclusion map of $A$ into $\cma$. This is
because  $A$ is $C^*$-nuclear if and only if the canonical map $A
\otimes D \to \cma \maxten D$ is a complete isometry with respect to
$\minten$ for every $C^*$-algebra $D$.

\medskip

The following result is a variant of the last theorem.

 \begin{proposition} \label{var}  If $A$ is an
exact approximately unital operator algebra then $A$ is
$C^*$-nuclear if and only if $C^*(F) \minten A = C^*(F) \maxten A$
completely isometrically for every discrete free group $F$.
\end{proposition}

\begin{proof}If $A$ is exact then the fact that any $C^*$-algebra
$B$ is a quotient of $C^*(F)$ for some $F$, forces exactness of the
sequence
\[ 0 \lra A \minten J \lra A \minten C^*(F) \lra A \minten B \lra
0.\] Applying Lemma \ref{maxq}, we have the exact sequence
\[ 0 \lra A \maxten J \lra A \maxten C^*(F) \lra A \maxten B \lra
0.\]  If $C^*(F) \minten A = C^*(F) \maxten A$ it follows that $ A
\minten B = A \maxten B$.
\end{proof}

\begin{corollary}\label{to}
  If $A$ is $C^*$-nuclear
and approximately unital, and if either $A$ is subexact or
$A$ is generated by unitaries, then $C^*_{\rm e}(A)$ is nuclear.
 \end{corollary}

\begin{proof} If $A$ is $C^*$-nuclear then so is $A^1$ by Corollary
\ref{w16}. Similarly if $A$ is subexact then so is $A^1$, since the
unitization of an exact $C^*$-algebra is exact. By Theorem
\ref{ewepcnuc} it follows that $C^*_{\rm e}(A^1) = C^*_{\rm e}(A)^1$
is nuclear and hence so is $C^*_{\rm e}(A)$.  On the other hand if
$A$ is generated by unitaries, then so is $C^*_{\rm e}(A)$, and so
$C^*_{\rm e}(A)$ is nuclear by \cite[Theorem 13.4]{Pisbook}.
\end{proof}

\section{Examples}

This main purpose of this section is to illuminate connections (or lack
thereof) of the properties studied above, in the case of some
extremely commonly encountered examples, to the matching
$C^*$-algebra properties for their $C^*$-covers.

\subsection{The disk algebra} \label{disk}
 It is well known (see e.g.\ \cite[6.2.5]{BLM}) that $A(\Ddb)$ is
$C^*$-nuclear.  Hence it has the AWEP and the HLLP, etc. We shall
show that $\cm{A(\Ddb)}$ has the LLP but is not nuclear (nor exact).
This shows amongst other things that the converse of the first
assertion in Proposition \ref{LLPd} fails.

 To see that $\cm{A(\Ddb)}$ is not exact, we will use the fact that
 $\cm{A(\Ddb)}$ is the
universal $C^*$-algebra generated by a contraction.  Let $B$ be any
separable $C^*$-algebra which is not exact. Since $\Kdb \minten B$
contains a complemented copy of $B$ and exactness, viewed as an
operator space property, would pass to this copy, it follows that
$\Kdb \minten B$ is not exact. By \cite{OlsZ}, $\Kdb \minten B$ is
singly generated as a $C^*$-algebra by a contractive element, call
it $x$. Since $\cm{A(\Ddb)}$ is the universal $C^*$-algebra
generated by a contraction, there exists a $*$-representation $\pi:
\cm{A(\Ddb)} \to \Kdb \minten B$ which is onto.  Since the exactness
is preserved by $C^*$-quotients it must be the case that
$\cm{A(\Ddb)}$ is not exact, and hence is not nuclear.

One may identify $\cm{A(\Ddb)}$ with $C^*_u<\Cdb>$, the universal
$C^*$-algebra for the operator space $\Cdb$, by comparing their
universal properties (see \cite{Oz1}).  With this identification we
can use \cite[Theorem 16.5]{Pisbook} to see that $\cm{A(\Ddb)}$ does
in fact have the LLP.

In \cite[Proposition 16.13]{Pisbook} one finds the remarkable fact
that $\cm{A(\Ddb)}$ having the WEP, is equivalent to Kirchberg's
important conjecture from \cite{Kir1} that WEP implies LLP.    It is
easy to argue that this is also equivalent to whether $\cm{A(\Ddb)}
\minten \cm{A(\Ddb)} = \cm{A(\Ddb)} \maxten \cm{A(\Ddb)}$.   Indeed
this follows from Kirchberg's remarkable result from \cite{Kir1}
that a $C^*$-algebra $B$ has WEP iff $B \minten B^{\rm op} = B
\maxten B^{\rm op}$; together with the fact that $\cm{A(\Ddb)} =
\cm{A(\Ddb)}^{\rm op}$.  The latter is a special case of the more
general fact that for an operator algebra $A$, $\cm{A^{\rm op}} =
\cma^{\rm op}$; whose proof is left as a simple exercise.

We remark that it is easy to see that  $$\cm{A(\Ddb) \maxten
A(\Ddb)} \neq \cm{A(\Ddb)} \maxten \cm{A(\Ddb)} .$$   Indeed by
(6.9) in \cite{BLM}, we have $A(\Ddb) \maxten A(\Ddb) = A(\Ddb)
\minten A(\Ddb)$ is the bidisk algebra (see \ref{bidi}). A pair of
commuting contractive representations of $A(\Ddb)$ gives rise to a
representation of $A(\Ddb) \maxten A(\Ddb)$, and hence to a
$*$-representation $\pi$ of $\cm{A(\Ddb) \maxten A(\Ddb)}$. It is
easy to construct an example of such representations (even two
dimensional, taking $z$ to $E_{21}$, where $z$ is the usual
generator of $A(\Ddb)$) such that $\pi(z \otimes 1)$ does not
commute with $\pi(1 \otimes \bar{z})$. On the other hand, these
would have to commute for any representation $\pi$ of $\cm{A(\Ddb)}
\maxten \cm{A(\Ddb)}$, by  \cite[Corollary 6.1.7]{BLM}.

\subsection{The bidisk algebra} \label{bidi}
 Using a result of Parrott \cite{Par} one may see
that the bidisk algebra $A(\Ddb^2)$ is $C^*$-nuclear (see e.g.\ p.\
266 in \cite{BLM}).   Thus it has the AWEP and the HLLP, etc.  On
the other hand, it is easy to argue from the universal property of
$C^*_{\rm max}$ applied to obvious maps between $A(\Ddb)$ and
$A(\Ddb^2)$, that $\cm{A(\Ddb)} \subset \cm{A(\Ddb^2)}$.  Hence
$\cm{A(\Ddb^2)}$ is not exact.  We do not know if $\cm{A(\Ddb^2)}$
has the LLP, or if $\cm{A(\Ddb)} \maxten \cm{A(\Ddb)}$ has the LLP.

\subsection{Triangular
matrices}   \label{tri} Let $T_n$ denote the $n \times n$ upper
triangular matrices, which are known to be $C^*$-nuclear \cite{PaP},
and hence it also has AWEP and the HLLP, etc. We will show that
$\cm{T_n}$ is not exact if $n \geq 3$, but it does have the LLP.
Note that  $\cm{T_2}$ is nuclear, since it can  be identified with
the subalgebra of $C([0,1],M_2)$ consisting of matrices which are
diagonal matrices at $t=0$  (see
 \cite[2.4.5]{BLM}).

For $n \geq 3$ we will first  show that $T_n$ is essentially a free
product of copies of $T_2$. Define for $1 \leq i \leq n-1$ the
algebra
\[ A_i := \underbrace{\mathbb{C} \oplus \cdots \oplus
\mathbb{C}}_{i-1 \mbox{ copies}} \oplus T_2 \oplus
\underbrace{\mathbb{C}  \oplus \cdots \oplus \mathbb{C}}_{n-i-1
\mbox{ copies}} \subseteq M_n.\] We will denote by $\iota_i$ the
inclusion map of $A_i$ into $M_n$ and we will let $D$ be the
subalgebra of $A_i$ given by the diagonal matrices.

\begin{lemma} The algebra $T_n$ is completely isometrically
isomorphic to $*_D A_i$. \end{lemma}

\begin{proof}
For each $i$ we have $A_i \subseteq T_n$ completely isometrically
 isomorphically.  It follows that there is a completely contractive
representation $*\iota_i$ of $*_DA_i$ into $T_n$. The range of
$*\iota_i$ contains a generating set for $T_n$ and hence the
representation $*\iota_i$ maps onto $T_n$. Next let $\pi: T_n \to
*_DA_i$ be given by letting
\[ \pi(E_{i,i}) \mapsto \underbrace{ 0 \oplus \cdots
\oplus 0}_{i-1 \mbox{ copies}} \oplus 1 \oplus \underbrace{0 \oplus
\cdots \oplus 0}_{n-i \mbox{ copies}} \in D
\] and \[ \pi(E_{i,i+1}) \mapsto \underbrace{ 0 \oplus \cdots \oplus
0}_{i-1 \mbox{ copies}} \oplus E_{1,2} \oplus \underbrace{ 0 \oplus
\cdots \oplus 0}_{n-i-1 \mbox{ copies}} \in A_i,\] where $E_{i,j}$
is the usual elementary matrix. It is easy to see that $ \pi$ is
well defined. Now extending using linearity and algebra operations
we have a representation. Notice that $ \pi \circ *\iota_i $ is
trivial on generators, as is $
*\iota_i  \circ \pi$, and hence these two maps are inverses.
The result will follow if we can show that $\pi$ is completely
contractive.
  This follows easily from the now standard
result from \cite{RONC} stating that it suffices to show that $\pi$
is contractive on matrix units.  But by construction,
$\pi(E_{i,i+k}) = \pi(E_{i,i+1})\pi(E_{i+1,i+2}) \cdots
\pi(E_{i+{k-1},i+k})$ which is a product of contractions and hence
is a contraction.  This is true for  $ 1 \leq i \leq n$ and $
1 \leq k \leq n-i$, and we are done.  \end{proof}

We now combine the last result with the fact that free products
`commute' with $C^*_{\rm max}$ (see \cite[Proposition 2.2]{BMax}),
to obtain
\[ \cm{T_n} = \underset{\scriptscriptstyle i}{\displaystyle{\ast}}\,
\cm{A_i} = \underset{\scriptscriptstyle i}{\displaystyle{\ast}}\,
\underbrace{\mathbb{C} \oplus \cdots \oplus \mathbb{C}}_{i-1 \mbox{
copies}} \oplus \cm{T_2} \oplus \underbrace{\mathbb{C} \oplus \cdots
\oplus \mathbb{C}}_{n-i-1 \mbox{ copies}} .\]

Since $\cm{T_n}$ is the $n$-fold free product of nuclear
$C^*$-algebras, by \cite[Theorem 13.2]{Pisbook} we have that
$\cm{T_n}$ has the LLP for all $n$.  The fact that $\cm{T_n}$ is not
exact for $n \geq 3$ follows from the next lemma, and the fact that
 $C([0,1])
*C([0,1])$ is not exact.  The latter is probably well known, but for the
readers convenience we give a proof, using facts in \cite[Section
5]{KiW} about the universal $C^*$-algebra $C^*_u(X)$ of an operator
system $X$.  Namely, $C^*_u(\ell^\infty_2) = C([0,1])$ and
$C^*_u(\ell^\infty_3)$ is not exact.   Define $\pi_k: \ell^\infty_2
\rightarrow \ell^\infty_3$ by $\pi_1((\lambda, \mu)) = ( \lambda,
\lambda, \mu )$ and  $\pi_2((\lambda, \mu)) = ( \lambda, \mu, \mu)$.
Clearly $\pi_k$ is a unital complete isometry for $k = 1, 2$. Notice
also that the ranges of the $\pi_k$ are jointly spanning.  By the
universal property for $C^*_u(\ell^\infty_2)$, there are unital
$*$-homomorphisms $\widetilde{\pi_k}: C([0,1]) \rightarrow
C^*_u(\ell^\infty_3)$ whose ranges together generate
$C^*_u(\ell^\infty_3)$.  By the universal property for free
products, there is a $*$-representation of $C([0,1])*C([0,1])$ onto
$C^*_u(\ell^\infty_3)$.  Since exactness passes to $C^*$-quotients,
it follows that $C([0,1]) *C([0,1])$ is not exact.

\begin{lemma} There is a completely isometric embedding of
the amalgamated free product $C([0,1])*C([0,1])$ into $\cm{T_n}$ for
$n \geq 3$.\end{lemma}

\begin{proof} Without loss of generality we will stick to the case
of $n=3$.  The proof $n > 3$ will follow in the same manner, or by
noting $\cm{T_3} \subset \cm{T_n}$.

We have $\cm{T_3} = (\cm{T_2} \oplus \Cdb) *_D \, (\Cdb \oplus
\cm{T_2})$.
Notice that the map $E_1:   \begin{bmatrix} f_{1,1} & f_{1,2} \\
f_{2,1} & f_{2,2} \end{bmatrix} \oplus \lambda \mapsto f_{2,2}$ is a
conditional expectation of $\cm{T_2} \oplus \Cdb$ onto a copy of
$C([0,1])$.  Similarly $E_2: \lambda \oplus \begin{bmatrix} g_{1,1}
& g_{1,2} \\ g_{2,1} & g_{2,2} \end{bmatrix}  \mapsto g_{2,2}$
defines a conditional expectation onto a copy of $C([0,1])$. Lastly
the map $d: D \to \mathbb{C}$ taking a $3 \times 3$ diagonal matrix to its
$2$-$2$ entry,
 is a conditional expectation
of $D$ onto a copy of $\Cdb$.   The result now follows from
\cite[Proposition 2.4]{ADEL}. \end{proof}

\subsection{The algebra $\mathcal{U}(X)$ for an operator space $X$} \label{uofx}

The operator algebra $\mathcal{U}(X)$ consists of $2 \times 2$ upper
triangular matrices with elements from $X$ in the $1$-$2$ corner and
scalars in the two diagonal entries; see 2.2.10-2.2.11 of
\cite{BLM}. We refer the reader to \cite[Section 6.4]{BLM} for a
discussion of the $\delta$ norm on tensor products.

\begin{lemma}\label{ux} If $X$ is an operator space, and if $\mathcal{U}(X)
\minten D = \mathcal{U}(X) \maxten D$ isometrically for a unital
$C^*$-algebra $D$, then $ \delta = {\rm min}$ on $X \otimes D$.  The
converse of this is true if $D \cong M_n(D)$ for all $n \in \Ndb$.
In particular, $X$ has the $1$-OLLP of Ozawa {\rm
\cite{Oz1,Pisbook}} if and only if $\mathcal{U}(X)$ is
$\Bdb$-nuclear.
\end{lemma}

\begin{proof}If $\Phi : X \to B(H)$ and $\pi : D \to B(H)$ are
commuting complete contractions, with $\pi$ a representation, then
by \cite[Proposition 2.2.11]{BLM}  we obtain a representation
$\theta_\Phi : {\mathcal U}(X) \to B(H^{(2)})$ which commutes with
$\pi^{(2)}: D \to B(H^{(2)})$. Thus if ${\mathcal U}(X) \minten D =
{\mathcal U}(X) \maxten D$ then for $x_k \in X, d_k \in D$, and with
$a_k$ the matrix with $x_k$ in its $1$-$2$ corner and $0$ elsewhere,
we have \[ \Vert \sum_k \; \theta_\Phi(a_k) \pi^{(2)}(d_k) \Vert =
\Vert \Phi(x_k) \pi(d_k) \Vert \leq \minnorm{ \sum_k \; a_k \otimes
d_k } = \minnorm{ \sum_k \; x_k \otimes d_k } .
\] Hence $\delta = {\rm min}$ on $X \otimes D$.

Conversely, suppose $\delta = {\rm min}$ on $X \otimes D$, and let
$\theta$ and $\rho$ be two commuting completely contractive
representations of ${\mathcal U}(X)$ and $D$ respectively. The
diagonal projections in ${\mathcal U}(X)$ induce a decomposition of
the Hilbert space as a sum $H \oplus K$ so that $\theta =
\theta_\Phi$ for a complete contraction  $\Phi : X \to B(K,H)$, and
$\rho(d) = \pi_1(d) \oplus \pi_2(d)$, for $d \in D$ and two
$*$-representations $\pi_1, \pi_2$ of $D$ on $H$ and $K$
respectively, such that $\Phi(x) \pi_2(d) = \pi_1(d) \Phi(x)$ for $x
\in X, d \in D$.  Note that $\theta_\Phi \circ c$ commutes with
$\rho$, where $c : X \to {\mathcal U}(X)$ is the canonical
embedding. It follows that with notation as above, for $x_k \in X,
d_k \in D$, \[ \Vert  \sum_k \; \theta_\Phi(c(x_k)) \rho(d_k) \Vert
= \Vert \sum_k \; \Phi(x_k) \pi_2(d_k) \Vert \leq \Vert \sum_k \;
x_k \otimes d_k \Vert_\delta = \minnorm{ \sum_k \; x_k \otimes d_k }
. \] Notice that $W = \Phi(X) \pi_2(D)$ is an operator $D$-bimodule,
a $D$-subbimodule of $B(K,H)$.    Also, $X \minten D$ is an operator
$D$-bimodule with the canonical actions. The computation above shows
that the map $u : x \otimes d \mapsto \Phi(x) \pi_2(d)$ is a
contractive $D$-bimodule map from $X \minten D$ to $W$.  If $D \cong
M_n(D)$, then we also have $\delta = {\rm min}$ on $X \otimes
M_n(D)$, and it is easy to see from this that $u$ is completely
contractive.  The map induced from the bimodule map $u$ by (3.12) of
\cite{BLM} is also completely contractive. One may argue from this
that for $b_1 , \cdots , b_n \in {\mathcal U}(X), d_1 , \cdots , d_n
\in D$:
\[ \Vert  \sum_k \; \theta(b_k) \rho(d_k) \Vert \leq \left| \left|
\left[ \begin{array}{ccl} \sum_k \; \lambda_k 1 \otimes d_k \; &
\sum_k \; x_k  \otimes d_k  \\ 0 \; & \sum_k \; \mu_k 1 \otimes d_k
\end{array} \right]
 \right| \right|
= \minnorm{ \sum_k \; b_k  \otimes d_k} .\] Here $\lambda_k , x_k ,
\mu_k$ are the three nonzero `corners' of $b_k$, and the norm of the
middle matrix is taken in $M_2(B(H \otimes K)$, where $X \subset
B(H), D \subset B(K)$. That is, ${\mathcal U}(X) \minten D =
{\mathcal U}(X) \maxten D$ isometrically.

We leave the remaining assertion to the reader. \end{proof}

\begin{corollary} \label{corux}  There exist unital operator algebras
$A$ with $C^*_{\rm e}(A)$ nuclear (hence having the LLP), but
$A$ is neither $\Bdb$-nuclear nor $C^*$-nuclear nor has the AWEP.
\end{corollary}

\begin{proof} If $X$ is a minimal operator space
without the $1$-OLLP, and if $A = \mathcal{U}(X)$, then  $A$ is not
$\Bdb$-nuclear by Lemma \ref{ux},  and hence not $C^*$-nuclear. But
$C^*_{\rm e}(A)$ is nuclear: it is a subalgebra of $M_2(B)$ for a
commutative $C^*$-algebra $B$ by \cite[Theorem 4.21]{BShi}, hence
Type I, and so nuclear. Since $A$ is exact but not $C^*$-nuclear it
cannot have AWEP by Theorem \ref{ewepcnuc}.
\end{proof}

From Lemma \ref{ux} and facts in \cite{Oz1}, it is easy to build
 $\Bdb$-nuclear operator algebras with bad properties.
Indeed for any finite dimensional operator space $X$ with $X^*$
$1$-exact, we have that $X$ has the $1$-OLLP, so that
$\mathcal{U}(X)$ is $\Bdb$-nuclear.

\medskip

{\bf Remarks.}  1)   \ We do not have an example of a $C^*$-nuclear
algebra $A$ with $C^*_{\rm e}(A)$ not nuclear, but presumably they
exist in abundance.  Equivalently, we do not know if
$C^*$-nuclearity implies subexactness.

2) \ For a given operator space it seems rather restrictive, and
therefore probably uninteresting, for $\mathcal{U}(X)$ to be
$C^*$-nuclear. Indeed, this is equivalent to saying that the
$\delta$ tensor norm agrees with the spatial (minimal) norm on $X
\otimes D$, for all $C^*$-algebras $D$.  Of course this occurs if $X
= \Cdb$, and if $X$ is a Hilbert row or column space (since $C_n
\otimes_{\rm h} D = C_n \minten D$ for any $C^*$-algebra $D$, and so
these also agree with $C_n \otimes_\delta D$, since $\delta \leq
{\rm h}$), but probably in few other cases.
We are indebted to Gilles Pisier and N. Ozawa for conversations on
this matter, which is related to the discussion on p.\ 341 of
\cite{Pisbook} of exact spaces whose dual is exact too.
If $X$ is finite dimensional with ${\mathcal U}(X)$ $C^*$-nuclear,
then as we saw in Lemma \ref{free}, ${\mathcal U}(X)$ is $1$-exact
and hence so is $X$. On the other hand, by \cite{Oz2}, since $X$ has
$1$-OLLP, $X^*$ is $1$-exact. There is only a small list of
$1$-exact finite dimensional spaces whose dual is known to be
$1$-exact too (see  p.\ 341 of \cite{Pisbook}).

\medskip

 We do not have an operator algebra version of Kirchberg's
profound characterization of separable exact $C^*$-algebras as
subalgebras of a fixed `universal' separable exact $C^*$-algebra
(see e.g.\ \cite{Kir1, Kir2, KPh}).  It seems feasible that there
does exist some such result, although the following rules out one
approach:

\begin{proposition} \label{nog} There exists a separable exact operator
space that is not linearly  completely isometric to a subspace of an
exact $C^*$-algebra.  There exists a separable unital exact operator
algebra which is not subexact.
\end{proposition}

\begin{proof} It is shown in \cite[Theorem 18]{KiW}
that there exists a separable exact operator system ${\mathcal S}$
that is not a unital-subsystem of any unital separable exact
$C^*$-algebra. Suppose that ${\mathcal S}$ was a subspace of an
exact $C^*$-algebra $A$. We will use notation from \cite{Ham} (see
also p.\ 285-286 of \cite{BShi}). Clearly $M_2(A)$ is exact, and
hence so too is its $C^*$-subalgebra generated by the copy of the `Paulsen
system' of ${\mathcal S}$. Since exactness also passes to
$C^*$-quotients, the $C^*$-envelope of the latter system is exact,
and hence so too is its upper right corner, the `ternary envelope'
of ${\mathcal S}$ (see \cite{Ham} or p.\ 286 of \cite{BShi}). This
envelope is completely isometric to $C^*_{\rm e}({\mathcal S})$, by
the uniqueness of the ternary envelope.  Thus $C^*_{\rm e}({\mathcal
S})$ is exact, with ${\mathcal S}$ as a unital subsystem,
contradicting the result cited from \cite{KiW} above.

For the last part, consider ${\mathcal U}({\mathcal S})$, an exact
separable operator algebra by the remarks above Proposition \ref{p},
which is not a subspace of an exact $C^*$-algebra. \end{proof}

\end{document}